\documentclass{article}
\usepackage{arxiv}
\usepackage{float}
\usepackage[utf8]{inputenc} 
\usepackage[T1]{fontenc}    
\usepackage{hyperref}       
\usepackage{url}            
\usepackage{booktabs}       
\usepackage{amsfonts}       
\usepackage{nicefrac}       
\usepackage{microtype}      
\usepackage{lipsum}
\usepackage{enumitem}
\usepackage{graphicx}
\usepackage{amsmath}
\usepackage[pagewise]{lineno}
\graphicspath{ {./images/} }
\usepackage{amsthm}
\newtheorem{thm}{Theorem}[section]
\theoremstyle{definition}
\newtheorem{prop}{Proposition}

\newtheorem{cor}{Corollary}[section]
\newtheorem{exmp}{Example}[section]

\title{On  chromatic parameters of some Regular graphs}

 \author{Prajnanaswaroopa S
}

\begin{document}

\maketitle

\section{Abstract}
In this work, we try to enunciate the Total chromatic number of some Cayley graphs like the Cayley graph on Symmetric group,  Alternating group, Dihedral group with respect to some generating sets and some other regular graphs.
\section{Introduction}
Total Coloring of a graph is the assignment of colors to the vertices and edges of the graph such that neither any two adjacent vertices, nor an edge and its incident vertices, nor any two adjacent edges receive the same color. The minimum number of colors required for total coloring of a graph is the total chromatic number of the graph. The Total Coloring Conjecture(TCC), proposed independently by Behzad and Vizing (\cite{BEH}, \cite{VIZ}), is a long-standing conjecture on graph coloring which states that the maximum number of colors required in a Total coloring is $\Delta+2$, where $\Delta$ is the maximum degree of the graph. The conjecture has been proved/verified true for a variety of classes of graphs. Some classical results are compiled in a great detail in the book by Yap(\cite{YAP}).  It has been proved for all classes of planar graphs with $\Delta\neq6$(\cite{BOR},\cite{SANZ}). It is also proved for graphs with $\Delta=3,4,5$(\cite{KOS1},\cite{KOS2},\cite{ROS},\cite{VIJ}) and also for $\Delta>\frac2{3}n$, where $n$ is the number of vertices of the graph(\cite{XIE1},\cite{XIE2}).  \\
A Cayley graph defined on a group $G$ with respect to a symmetric generating set $S$ is the graph which has all the elements of the group as its vertices, and two vertices $x$ and $y$ in the graph are adjacent if and only if $x=ys$ for some $s\in S$. The symmetric nature of the generating set implies that if $s\in S\implies s^{-1}\in S$. The generating set $S$ has the additional property that the identity of the group $G$ does not belong to it.\\ In the realm of Cayley graphs, very less progress has been made in proving/verifying the TCC. Specifically, many papers focus on establishing the conjecture for Cayley graphs on cyclic groups mainly, or Circulant graphs (\cite{KHE}, \cite{NIG}, \cite{JUN}, \cite{CAM}).
In this paper, we consider the total coloring of Cayley graphs on some other groups, which are non-cyclic. \\The symmetric group of order $n$ is the group of all bijective functions from an $n$-element set to itself, with the function composition as the group operation. The elements of the symmetric group are written in the bracket notation. The elements which have exactly two symbols in their bracket notation are said to be transpositions. The set of all transpositions on an $n$-element set generate the whole symmetric group of order $n$. A minimal generating set for the symmetric group of order $n$ is the transpositions of the form $(12), (13), \ldots, (1n)$ (\cite{CON}).
\section{Cayley graphs of Permutation groups with transposition generators}
We denote the Cayley on a group $G$ to be $C(G, S)$, where $S$ is a generating set of the group as well as graph. The Symmetric group on $n$ symbols will be denoted by $S_n$. The set of minimal transposition generators in $S_n$ will be denoted by $T_m=\{(12),(13),\ldots,(1n)\}$. The set of all transpositions in $S_n$ will be denoted by $T$. 
\begin{prop}
If $n$ is odd and $n>2$, then the vertices of the graph $G=C(S_n, T_m)$ can be $3$-colored.
\end{prop}
\begin{proof}
We first start with any element, say the identity element, and give it a color, say $A$. Next we distribute the elements adjacent to the vertices $A$ in two other color classes $B$ and $C$ evenly. Similarly, we distribute the neighbors of any other color class evenly in other color classes. This is possible always because:\\
i) We have  $\Delta=n-1$ is even\\
ii) Product of three elements (any simultaneous two of which are distinct, though the elements need not all be distinct) of the generating set elements does not give us another generating set element. Restated, this means that no two elements in the same independent set having their minimal transposition decomposition (the number of  transpositions from the set $T_m$ required whose product equals that transposition) differing greater than or equal to two transpositions cannot be adjacent.\\
\end{proof}
\begin{thm}
If $n>2$, then the total chromatic number of the graph $G=C(S_n, T_m)$ is $n$
\end{thm}

\begin{proof} We proceed by finding a $3$- divisible coloring of the vertices of $S_n$, by using orbits of the graph vertices corresponding to $S_3$. We do this by using induction. First, consider coloring the elements of $S_4$. For this, the graph corresponding to $S_3$ is colored into independent sets as $[e,(23)], [(12)(132)], [(13), (123)]$. We then take the  orbits of this subset with respect to the group with respect to left multiplication with the elements of the form $(12)(14), (13)(14)$ in the first independent set to get $18$ elements of $S_4$ arranged into independent sets (The element $(12)(14), (13)(14)$  and left product of $(12)(14), (13)(14)$ with $(23)$ are   placed in the first independent set, these acting as seed elements for the orbit generation with respect to left multiplication; the rest of the orbit elements are placed in the respective independent sets). The remaining $6$ elements are colored greedily. 

We assume we have arranged the elements of $S_{n-1}$ into independent sets in this manner. The $S_n$ elements are then arranged into independent sets   by  using the orbits formed by left multiplications of $S_{n-1}$ with the elements of the second and third independent sets of $S_{n-1}$ right multiplied by $(1 n)$ (The element $h_1$ and the left product of $g_1$ with $(23)$, where $h_1=g_1(1 n)$ and $g_1$ is an element in the second/ third independent set; are   placed in the first independent set; the rest of the orbit elements are placed in the respective independent sets). This gives us first set of $(n-1)!$ extra elements colored(apart from the $(n-1)!$ elements of $S_{n-1}$ already colored). The next set of $(n-1)!$ elements is produced by forming orbits of left multiplication of the orbit  obtained in the second step  right multiplied by $(12)$. Similarly, the next set of elements is produced by forming orbits with respect to left multiplication with the graph corresponding to $S_{n-1}$ by the elements in the second and third independent sets in the previous step right multiplied by $(12\ldots n)$. The process is then continued till we get all the $n!$ elements colored. The resulting $3$-coloring of vertices can then be converted to total coloring using the edges of $1$-factor between any two independent sets in the third independent set (which is guaranteed by the coloring process above). This gives us a partial $3$-total coloring. The remaining edges can be  $1$- factorized as the graph is bipartite. This gives us a full total coloring of $S_n$.  
\end{proof}
\begin{exmp}
The $3$-coloring of the elements of $S_4$ can be achieved by dividing the elements into three independent sets as $[e, (23), (124), (1324), (134),(1234), (142), (1432)],$ \\
$[(12), (132), (24), (243),(1342), (12)(34),(14), (14)(23) ]$,\\
$ [(13), (123), (1243), (13)(24), (34), (234), (143), (1423) ] $. This is then extended to a total coloring of $C(S_4, T_m)$ by expanding the perfect matchings between any two independent sets in the corresponding third set. Note that the independent set division is done using the algorithm described in the proof.
\end{exmp}
\begin{thm}
The Cayley graph $G=C(A_n, S)$  where $S$ is the generating set of three cycles of the form $\{(123),(124),\ldots,(12n),(1n2),\ldots,(142),(132)\}$ on the Alternating group $A_n$ satisfies TCC.
\end{thm}
\begin{proof}
 By following a similar procedure as in the  Proposition 2.1, that is by distributing the neighbors of any vertex equitably in two different color classes, we can find a $3$- divisible total coloring, which can then be extended to the whole graph by edge coloring the remaining edges. The  $3$- coloring is possible owing to the fact that the degree $2(n-2)$ is always even. Care has to be taken when dividing the neighbors of  vertex into two different color classes, as there are adjacencies between the neighbors of any vertex. This is achieved by putting vertices which belong to the same triangle in different color classes. We observe that the vertices $e, (123), (132)$ belong to the same triangle; so we give a $3$-coloring to these vertices. Now, by a similar analogy we give $3$- coloring to $e, (124), (142)$; or generally $e,(12n), (1n2)$ are always placed in $3$ different color classes. Note that none of $12n$ is adjacent to any of $(12m)$ and similarly none of $(1n2)$ is adjacent $(1m2)$, thus they can be placed in the same color classes. Now, we distribute the neighbors of each of the already colored vertices into the two different color classes of which the vertices are not part of. This is possible because of none of $(12n_1)(12m_1)$ would adjacent to $(12n_2)(12m_2)$ where $m_1, n_1\neq m_2, n_2$. Similarly, we repeat the procedure of the new elements obtained thus, coloring all the vertices into three colors. Note that any two color classes have a distinct perfect matching in the above coloring as we have distributed the neighbors equally into two different color classes. The perfect matchings between each of two color classes is then put in the third color class thus giving us a partial total coloring of $G$. This partial total coloring is then extended to the whole of $G$ by coloring the remaining edges which can be done in $n-4+1=n-3$ colors by the Vizing's theorem. Hence, we can give a total coloring of $G$ in $3+(n-3)=n$ colors.    
\end{proof}
\begin{thm}
The Cayley graph $C(S_n,S)=G$ on the group $S_n$ with \\
$S=\{(12), (12\ldots n),(n n-1\ldots 2 1)\}$  is type I.
\end{thm}
\begin{proof}
The graphs with $3|n$ have an odd cycle and thus are not bipartitie. The rest of the graphs are bipartite. In any case, TCC is clear. To prove type I, we proceed by finding a $3$- divisible coloring of the vertices, by using orbits of left multiplication of the graph vertices corresponding to $S_3$. We do this by using induction. First, consider the elements of $S_4$. For this, the graph corresponding to $S_3$ is colored into independent sets as $[e,(23)], [(12)(132)], [(13), (123)]$. We then take the  orbits of left multiplication of this subset with respect to the group  elements of the form $(12)(1234), (13)(1234)$ in the first independent set to get $18$ elements of $S_4$ arranged into independent sets (The element $g(12)(1234)$  and left product of $(13)(1234)$ with $(23)$ are   placed in the first independent set; the rest of the orbit elements are placed in the respective independent sets). The remaining $6$ elements are colored greedily. We assume we have arranged the elements of $S_{n-1}$ into independent sets in this manner. The $S_n$ elements are then arranged into independent sets   by  taking the orbits of left multiplication of $S_{n-1}$ with the elements of the second and third independent sets of $S_{n-1}$ right multiplied by $(12\ldots n)$ (The element $h_1$ and the left product of $g_1$ with $(23)$, where $h_1=g_1(12\ldots n)$ and $g_1$ is an element in the second/ third independent set; are   placed in the first independent set; the rest of the orbit elements are placed in the respective independent sets). This gives us first set of $(n-1)!$ extra elements colored(apart from the $(n-1)!$ elements of $S_{n-1}$ already colored). The next set of $(n-1)!$ elements is produced by forming  orbits of left multiplication using the elements obtained in the second step above right multiplied by $(12)$. Similarly, the next set of elements is produced by forming orbits of left multiplication with the graph corresponding to $S_{n-1}$ by the elements in the second and third independent sets in the previous step right multiplied by $(12\ldots n)$. The process is then continued till we get all the $n!$ elements colored. The resulting $3$-coloring of vertices can then be converted to partial total coloring using the edges of $1$-factor between any two independent sets in the third independent set (which is guaranteed by the coloring process followed above). This gives us a partial $3$-total coloring. The remaining edges form a $1$- factor. This gives us a full total coloring of $G$.  
\end{proof}
\begin{thm}
The Cayley graph $G=C(A_n,S$ on the group $A_n$ with \\ $S=\{(123), (132), (12\ldots n),(n n-1\ldots 2 1)\}$ for $n\ge4$  is type I.
\end{thm}
\begin{proof}
 We proceed by finding a $3$- divisible coloring of the vertices, by using orbits of the total coloring of the graph corresponding to $S_3$. We do this by using induction. First, consider the elements of $S_4$. For this, the graph corresponding to $S_3$ is colored into independent sets as $[e,(23)], [(12)(132)], [(13), (123)]$. We then take the orbits of this subset with respect to left multiplication by the elements the group of the form $(12)(1234), (13)(1234)$ in the first independent set to get $18$ elements of $S_4$ arranged into independent sets (The element $g(12)(1234)$  and left product of $(13)(1234)$ with $(23)$ are   placed in the first independent set; the rest of the orbit elements are placed in the respective independent sets). The remaining $6$ elements are colored greedily. We assume we have arranged the elements of $S_{n-1}$ into independent sets in this manner. The $S_n$ elements are then arranged into independent sets   by  taking orbits of left multiplication of $S_{n-1}$ with the elements of the second and third independent sets of $S_{n-1}$ right multiplied by $(12\ldots n)$ (The element $h_1$ and the left product of $g_1$ with $(23)$, where $h_1=g_1(12\ldots n)$ and $g_1$ is an element in the second/ third independent set; are   placed in the first independent set; the rest of the orbit elements are placed in the respective independent sets). This gives us first set of $(n-1)!$ extra elements colored(apart from the $(n-1)!$ elements of $S_{n-1}$ already colored). The next set of $(n-1)!$ elements is produced by forming orbits of left multiplication using the elements obtained in the second step above right multiplied by $(12)$. Similarly, the next set of elements is produced by forming orbits of left multiplication with the graph corresponding to $S_{n-1}$ by the elements in the second and third independent sets in the previous step right multiplied by $(12\ldots n)$. The process is then continued till we get all the $n!$ elements colored, by the principle of induction. The resulting $3$-coloring of vertices can then be converted to partial total coloring using the edges of $1$-factor between any two independent sets in the third independent set (which is guaranteed by the coloring process above). This gives us a partial $3$-total coloring. The remaining edges can be  $1$- factorized. This gives us a full total coloring of $G$.  
\end{proof}

\section{Cayley Graphs on Dihedral Groups}
The Dihedral groups will be denoted by $D_{2n}$. Its minimal generators are denoted by $\{r,s\}$, where $r$ is the reflection element having order $2$ and $s$ is the rotation element having order $n$ with the usual group defining relations $(rs)^2=e$. Note that the powers of the rotation element form a cyclic group and its elements can be denoted by $\{1,2,\ldots,n\}$.
\begin{thm} Let $n$ be even. If $G$ is a Cayley graph on the Dihedral group $D_{2n}$ with generating set $\{r, 1,2,\ldots,k, n-k,\ldots, n-2, n-1\}$ satisfies TCC.
\end{thm}
\begin{proof}
The vertices of the full graph consist of two parts- the vertices $e, s, s^2,\ldots, s^{n-1}$ and the vertices $r, rs, rs^2, \ldots rs^{n-1}$. In order to color the vertices, we use the vertex coloring of the vertices $\{e, s, s^2, \ldots, s^i,\ldots, s^{n-1}\}, i\in\mathbb{N},\, 0\le s\le n-1$ (this set is can be also written as $0,1,2,\ldots, n-1$) induced from the total coloring of the same set of vertices. By using the result of Theorem 16 in Campos and de Mello (\cite{CAM}), we get a coloring of the vertices using at most $2k+2$ colors, which we denote $\chi$. We can extend this vertex coloring to include all the vertices of the graph by taking left orbits of each independent set of vertices previously formed with respect to left multiplication with an element of the form  $rs^j\,,j\in\mathbb{N},1\le j\le n-1$. The orbits of one independent set are given the same color as any other independent set(other than the set of which the orbit was taken). Typically, for sake of convention, we give the orbits of the $i$-th independent set the same color as the $i+1$-th (modulo $\chi$).   This gives us the vertex coloring of all the graph vertices using at most $\chi$ colors. As for the edges of the graph, we can split the edges of $G$ into two  parts-\\
i) The edges induced by the set $S_1=\{1,2,\ldots, k, n-2,n-1\}$ and\\
ii)The edges induced by $r$.
The edges induced by $S_1$ form two copies of $k$-th powers of $n$-cycle. This part is colored similarly to the vertices by induction from the total coloring of this set of vertices which satisfies TCC from the result of Theorem 16 in Campos and de Mello (\cite{CAM}). The edges induced by $r$ form a connecting link between the two copies of powers of $n$-cycle. It is actually a $1$-factor, which can be easily edge colored with an extra color. Hence, the total number of colors required will not exceed $2k+3$ which verifies TCC for this class of graphs.
\end{proof}
\begin{exmp}
Consider the group $D_{72}$ with generating set $\{1, 2,3, 4, 32, 33, 34, 35, r\}$. Then, we have $\chi=9$, which can be extended to color all vertices of the Cayley graph on $D_{72}$.
\end{exmp}

\begin{thm}
Let $T=T_1\cup T_2$ where $T_1\subset\{0,1,2,\ldots, n-1\}$. Let the circulant graph of order $n$ with generating set $T_1$ be total colorable with $|T|+2$ colors and each total independent set having  vertices with the same difference (the difference between any pairs of vertices in an independent set is the same as that for the corresponding pair in any other independent set), and $T_2=\{rs^i\}$ where $i$ belongs to any subset of $\{0,1,2,\ldots,n-1\}$  . Then, the Cayley graph on the Dihedral group $D_{2n}$ with respect to the generating set $T$ satisfies TCC.
\end{thm}
\begin{proof}
Here, let us assume $s=1$. Then, we could the elements $rs^a$  as $ra. $We first divide the vertices $\{0,1,\ldots,n-1\}$ into the $|T_1|+2$ independent sets. Suppose, in addition, the first independent set has the vertices $[0, a_1, a_2, a_3, \ldots a_k]$ Now, we take orbits of left multiplication of the vertices colored previously with respect to an element $rj,\,j\in\{0,1,\ldots,n-1\}$ such that $j\neq i,i+d$ where $d=|a_i-a_j|$ for $a_i$ in the first independent set. This would give us the vertex coloring of all vertices in the graph using $|T_1|+2$ colors, because, on taking orbits, each of the adjacencies of the vertices in the coloring of circulant graph are excluded by the condition imposed on $j$.  Since the remaining edges generated by the elements $rs^i$ are $1$-factor, we could finish the total coloring of the whole graph with one extra  color thereby giving us a total coloring with $|T|+2$ colors which then verifies the TCC for this class of graphs. 
\end{proof}
\begin{exmp}
Consider the group $D_{36}$ with generating set $\{1, 2,3, 4, 14, 15, 16, 17, rs^2,r\}$. Then, we have $\chi=9$, which can be extended to color all vertices of the Cayley graph on $D_{36}$.
\end{exmp}
\begin{thm}
Let $T=T_1\cup T_2$ where $T_1\subset\{0,1,2,\ldots, n-1\}$. Let the circulant graph of order $n$ with generating set $T_1$ be total colorable with at most $|T_1|+2$ colors and each total independent set having  at most $k$ vertices such that any two consecutive pairs of vertices in an independent set have the same difference $d$, and $T_2=\overline{T_1}-\{ri,r(i+d),\ldots, r(i+kd-k)\}$. Then, the Cayley graph on the Dihedral group $D_{2n}$ with respect to the generating set $T$ satisfies TCC.
\end{thm}
\begin{proof}
Let us index the independent sets of the vertex coloring of the total colorable circulant graph of order $n$ with generating set $T_1$ as $I_j$, where $j$ ranges from $0$ to  $|T_1|+2$. We first color the vertices $\{0,1,2,\ldots, n-1\}$ as in a total coloring of the circulant graph of order $n$ with generating set $T_1$. We extend this vertex coloring to the vertices of the full graph by forming orbits of the prior independent with elements of the form $r(i-j)$ where $j$ is the index of the independent set $I_j$. This will give us  independent sets of the full graph. This is by virtue of the generating set $T_2$ excluding elements of the form $\{ri, r(i+d), \ldots r(i+kd-k)\}$ whereby the adjacencies of any elements in the prior vertex coloring are excluded. Now, as regards the edge coloring, we use the similar edge coloring for the edges among the identical copies of vertices $\{0,1,2,\ldots,n-1\}$ and $\{r, r1,r2\ldots,r(n-1\}$. For edges between the above two sets of vertices, as in the previous two theorems, we give one color for each generated by each element in $T_2$. Therefore,  we totally use at the most $|T_1|+|T_2|+2=|T|+2$ colors to totally color the graph, which  thereby verifies TCC. 
\end{proof}
\begin{exmp}
Let $n=8k+4$. Consider the graph on $D_{2n}$ with the generating set $T=T_1\cup T_2$ with $T_1\{0, 1, 2,\ldots,k,n-k,\ldots n-2,n-1\}$, $T_2=D_{2n}-\{r(2k+1), r(4k+2),r(6k+3),r(8k+4)\}$. We see that the circulant graph of order $n$ with the generating set $T_1$ is the power of cycle graph $C_{8k+4}^k$, which therefore total colorable with $2k+1$ colors. By using the process as described in the theorem, this total coloring could be extended to the whole graph, whereby we can color the graph using $|T|+1=2n-2k-2$ colors, which is actually a type I total coloring.
\end{exmp}
\section{Complement of Kneser Graphs}
The Kneser Graphs $K(n,k)$ are graphs consisting of $\binom{n}{k}$ vertices which are $k$-subsets of an $n$ element set, with two vertices adjacent whenever those correspond to disjoint sets. \\
A Hypergraph is a generalization of the usual simple graph. It consists of $n$ vertices with hyperedges corresponding to $k_i$-subsets of the $n$ element set where $i$ varies through an index set. The hypergraph is uniform if the hyperedges all have the same cardinality, that is $k_i=k$ is constant. In addition, if all the $k$-element sets are present in the hyperedges, then the hypergraph is said to be complete, here denoted by $H(n,k)$. It can be seen that the line graph of a complete hypergraph corresponds to the complement of Kneser Graph. In particular, the complement of the line graph of a complete simple graph of order $5$ is the Petersen graph.   
\begin{thm}
If $n$ be even and $k|n$, then the complement of the Kneser Graph $K(n,k)$, $G=\overline{K(n,k)}$ satisfies $\chi''(G)\le\Delta+3$.
\end{thm}
\begin{proof}
By Baranayai's theorem (\cite{BAR}), we have that the hyperedges of a complete hypergraph $H(n,k)$ can be factorized evenly into $\frac{\binom{n}{k}}{\frac{n}{k}}$ classes. This implies  we can partition the vertices of $\overline{K(n,k)}$ into $\frac{n}{k}$ disjoint cliques having $\frac{\binom{n}{k}}{\frac{n}{k}}$ vertices, which thus gives a $\frac{n}{k}$ coloring of the vertices. Now, to color the edges, we use the coloring of the complete graph of order $\frac{n}{k}$. We first canonically color the vertices and edges of each clique of order $\frac{n}{k}$ using the canonical total coloring of a clique of order $\frac{n}{k}$. This gives a partial $\frac{n}{k}$ or $\frac{n}{k}+1$-total coloring of $\overline{K(n, k)}$. We then extend it  using extra colors to  color the connecting edges between the cliques. Since at most $\frac{n}{k}+1$ colors are required to color all the disjoint cliques, by Vizing's theorem, we require at most $\binom{n-k}{k}-\frac{n}{k}+2$ colors to color the connecting edges joining the cliques, the total number of colors required is at most $\binom{n-k}{k}-\frac{n}{k}+2+\frac{n}{k}+1=\binom{n-k}{k}+3=\Delta+3$ colors. Hence, $\chi''(G)\le \Delta+3$.
\end{proof}
\begin{cor}
If, in the above theorem, $\frac{n}{k}$ be odd, then $\overline{K(n,k)}$ satisfies TCC. In addition,  if the remaining edges could be colored using $\delta$ colors, where $\delta$ be the degree of the graph induced by the connecting edges between the cliques of order $\frac{n}{k}$, then the graph $\overline{K(n,k)}$ is type I.
\end{cor}
\begin{proof}
The proof is an immediate consequence of the previous theorem, once we know that the clique of odd order $\frac{n}{k}$ requires just $\frac{n}{k}$ colors for its total coloring.
\end{proof}


\begin{thebibliography}{99}
\bibitem{BAR}Baranyai, Zsolt. "The edge-coloring of complete hypergraphs I." Journal of Combinatorial Theory, Series B 26.3 (1979): 276-294.
\bibitem{BEH}Behzad, Mehdi. Graphs and their chromatic numbers. Michigan State University, 1965.
\bibitem{BOR} Borodin, O.V.  On the total coloring of planar graphs, J. Reie Angew. Math., 394: 180-185, 1989.
\bibitem{CAM}Campos, C. N., and Célia Picinin de Mello. "A result on the total colouring of powers of cycles." Discrete Applied Mathematics 155.5 (2007): 585-597.
\bibitem{CON}Conrad, Keith. "Generating sets." Expository, unpublished paper on the author’s personal homepage (2013).
\bibitem{JUN}Junior, Mauro N. Alves, and Diana Sasaki. "A result on total coloring of circulant graphs." Anais do V Encontro de Teoria da Computação. SBC, 2020.
\bibitem{KHE}Khennoufa, Riadh, and Olivier Togni. "Total and fractional total colourings of circulant graphs." Discrete mathematics 308.24 (2008): 6316-6329
\bibitem{KOS1} Kostochka, A.V. The total coloring of a multigraph with maximum degree 4, Discrete Math., 17(2):
161-163, 1977.
\bibitem{KOS2} Kostochka, A.V.  The total chromatic number of any multigraph with maximum degree five is at most
seven, Discrete Math., 162(1-3): 199-214, 1996
\bibitem{NIG}Nigro, Mauro, Matheus Nunes Adauto, and Diana Sasaki. "On total coloring of 4-regular circulant graphs." Procedia Computer Science 195 (2021): 315-324.

\bibitem{ROS} Rosenfeld, M.  On the total coloring of certain graphs, Israel J. Math., 9: 396-402. 1971
\bibitem{SANZ}Sanders, D.P., and Zhao, Y.  On total 9-coloring planar graphs of maximum degree seven, J. Graph
Theory, 31(1): 67-73, 1999.
\bibitem{STO}Stong, Richard A. "On 1-factorizability of Cayley graphs." Journal of Combinatorial Theory, Series B 39.3 (1985): 298-307.
\bibitem{VIJ} Vijayaditya, N. On total chromatic number of a graph, J London Math Soc. 3 (1971), 405-408
\bibitem{VIZ}Vizing, Vadim G. "Some unsolved problems in graph theory." Russian Mathematical Surveys 23.6 (1968): 125.
\bibitem{XIE1}Xie, Dezheng, and Zhongshi He. "The total chromatic number of regular graphs of even order and high degree." Discrete mathematics 300.1-3 (2005): 196-212.
\bibitem{XIE2}Xie, DeZheng, and WanNian Yang. "The total chromatic number of regular graphs of high degree." Science in China Series A: Mathematics 52.8 (2009): 1743-1759.
\bibitem{YAP}Yap, Hian Poh. Total colourings of graphs. Springer, 2006.
\end{thebibliography}
\end{document}